\documentclass[reqno]{amsart}

\usepackage{a4wide}
\usepackage{color}
\usepackage{mathrsfs}
\usepackage{mathtools}
\usepackage{tikz-cd}
\usepackage{amsmath}
\usepackage{amssymb}
\usepackage{bbm}
\numberwithin{equation}{section}
\usepackage[colorlinks,citecolor=green,linkcolor=red]{hyperref}
\usepackage{hyphenat}
\usepackage{lmodern}

\usepackage[latin1]{inputenc}

\newcommand{\B}{\mathbb{B}}
\newcommand{\N}{\mathbb{N}}
\newcommand{\R}{\mathbb{R}}
\newcommand{\sfd}{{\sf d}}
\renewcommand{\d}{{\mathrm d}}

\newcommand{\restr}[1]{\lower3pt\hbox{\(|_{#1}\)}}

\newcommand{\nchi}{{\raise.3ex\hbox{\(\chi\)}}}

\newcommand{\fr}{\penalty-20\null\hfill\(\blacksquare\)}
\newcommand{\X}{{\rm X}}

\newcommand{\LIP}{{\rm LIP}}
\newcommand{\Lip}{{\rm Lip}}
\newcommand{\lip}{{\rm lip}}

\newcommand{\Cyl}{{\rm Cyl}}

\newtheorem{theorem}{Theorem}[section]
\newtheorem{corollary}[theorem]{Corollary}
\newtheorem{lemma}[theorem]{Lemma}

\newtheorem{remark}[theorem]{Remark}

\title[Smooth approximations preserving asymptotic Lipschitz bounds]
{Smooth approximations preserving \\ asymptotic Lipschitz bounds}
\author{Enrico Pasqualetto}
\address{Department of Mathematics and Statistics,
P.O.\ Box 35 (MaD), FI-40014 University of Jyvaskyla}
\email{enrico.e.pasqualetto@jyu.fi}

\begin{document}
\date{\today} 
\begin{abstract}
The goal of this note is to prove that every real-valued Lipschitz function on a Banach space can
be pointwise approximated on a given \(\sigma\)-compact set by smooth cylindrical functions whose
asymptotic Lipschitz constants are controlled. This result has applications in the study of metric
Sobolev and BV spaces: it implies that smooth cylindrical functions are dense in energy in these
kinds of functional spaces defined over any weighted Banach space.
\end{abstract}
\maketitle
\section{Introduction}
Smooth approximations of continuous and Lipschitz functions on Banach spaces have been deeply investigated, see
e.g.\ \cite{DGZ93,FMZ06,FHHMZ11,HJ14} and the references therein for an account of the vast literature about this topic.
In this note, we establish a quantitative smooth approximation result for real-valued Lipschitz functions
defined on a Banach space \(\B\). The approximants are \emph{smooth cylindrical functions}, which form an algebra of infinitely-differentiable
functions (in the Fr\'{e}chet sense), see \eqref{eq:def_Cyl}. The approximation is achieved on a \(\sigma\)-compact subset of \(\B\),
with a pointwise control on the \emph{asymptotic slopes} \(\lip_a(f_n)\) (that associate to \(x\in\B\) the asymptotic Lipschitz constant
of \(f_n\) at \(x\), see \eqref{eq:def_lip_a}) along the approximating sequence \((f_n)_n\).
Specifically, in Theorem \ref{thm:approx_with_cyl} we prove that, given a bounded Lipschitz function \(f\colon\B\to\R\) and a \(\sigma\)-compact
set \(E\subseteq\B\), there exists an equi-bounded equi-Lipschitz sequence \((f_n)_n\) of smooth cylindrical functions \(f_n\colon\B\to\R\) with
\begin{equation}\label{eq:approx_prop}
\lim_n f_n(x)=f(x),\qquad\varlimsup_n\lip_a(f_n)(x)\leq\lip_a(f)(x)
\end{equation}
for every \(x\in E\). On the one hand, the approximation is obtained just on \(\sigma\)-compact sets, but on the other hand the result
is valid for arbitrary Banach spaces. The proof of Theorem \ref{thm:approx_with_cyl} combines the fact that each Banach space
is contained in a Banach space having the \emph{metric approximation property} (Corollary \ref{cor:embed_in_MAP})
with an extension result by Di Marino--Gigli--Pratelli \cite{DMGP20}, which ensures that any Lipschitz
function can be extended preserving its asymptotic Lipschitz constants. As a consequence of Theorem \ref{thm:approx_with_cyl},
if \(\mu\) is a Radon measure on \(\B\), then for any bounded Lipschitz function \(f\colon\B\to\R\) there is an equi-bounded
equi-Lipschitz sequence \((f_n)_n\) of smooth cylindrical functions \(f_n\colon\B\to\R\) such that \eqref{eq:approx_prop} holds for
\(\mu\)-a.e.\ \(x\in\B\). Cf.\ the last paragraph of Section \ref{s:approx_result}.
\medskip

The above approximation result is tailored to the study of metric Sobolev and BV spaces defined over a
weighted Banach space. In metric geometry and in nonsmooth analysis, first-order \emph{Sobolev spaces} \(W^{1,p}(\X,\mu)\)
of exponent \(p\in[1,\infty)\) (see \cite{Cheeger00,Shanmugalingam00,AmbrosioGigliSavare11,AmbrosioGigliSavare11-3,
DiMarino14,EB20,AILP24}) and the space \(BV(\X,\mu)\) of \emph{functions of bounded variation}
(see \cite{Mir03,ADM2014,DiMarino14,Martio16,NobiliPasqualettoSchultz23}) over a metric measure space
\((\X,\sfd,\mu)\) play a fundamental role. A distinguished class of metric measure spaces is the one of
\emph{weighted Banach spaces}, i.e.\ separable Banach spaces \(\B\) equipped with a finite Borel measure \(\mu\).
In the specific setting of weighted Euclidean spaces, notions of Sobolev and BV spaces were introduced and studied in \cite{BBS97,BBF99,Zhi00}
-- prior to the development of a metric Sobolev and BV calculus -- with applications in the analysis of several variational problems,
e.g.\ shape optimisation \cite{BB01,BL22,BB22}, optimal transport with gradient penalisation \cite{Louet14} and homogenisation \cite{Zhi02,HM17}.
The consistency between the notions in \cite{BBS97,BBF99,Zhi00} and the metric theory was proved in \cite{Gig:Pas:21,LPR21,GL22}. Recently, also Sobolev spaces of
exponent \(p\in(1,\infty)\) over (infinite-dimensional or non-Hilbertian) weighted Banach spaces were investigated in \cite{Pa:Ra:23} (for the reflexive case)
and in \cite{LP24} (for the general case), while \(W^{1,1}\) and \(BV\) have not been studied in this framework yet. Informally, on any weighted Banach space \((\B,\mu)\)
smooth cylindrical functions are \emph{dense in energy} in \(W^{1,p}(\B,\mu)\) for all \(p\in[1,\infty)\) and in \(BV(\B,\mu)\); we will prove it as a corollary of
Theorem \ref{thm:approx_with_cyl}, see Theorem \ref{thm:dens_cyl} for the precise statement. Whereas this result was previously known for \(p>1\) (it follows from
\cite[Theorem 5.2.7]{Sav19}, see also \cite{FornasierSavareSodini22,Sodini22,LP24}), it is new for \(W^{1,1}(\B,\mu)\) and \(BV(\B,\mu)\).
The density in energy of smooth functions allows to transfer geometric or analytic information (e.g.\ Hilbertianity, reflexivity
or uniform convexity) from the Banach space \(\B\) to the Sobolev and BV spaces over \((\B,\mu)\), due to the regular behaviour of the
asymptotic slope \(\lip_a(f)\) of a smooth function \(f\colon\B\to\R\) (which coincides at \(x\in\B\) with the norm of the Fr\'{e}chet differential
of \(f\) at \(x\)). See for example \cite[Corollary 5.3.11]{Sav19} or \cite[Corollary 5.5]{Sodini22}.
\subsection*{Acknowledgements}
The author wishes to thank Miguel Garc\'{i}a-Bravo for the useful discussions on the topics of this paper.
The author acknowledges the support by the MIUR-PRIN 202244A7YL project ``Gradient Flows and Non-Smooth Geometric
Structures with Applications to Optimization and Machine Learning'' and by the Research Council of Finland (project
``Functions of bounded variation on metric-measure structures'', grant 362898).
\section{Preliminaries}
\subsection{Metric spaces}
Given a metric space \((\X,\sfd)\), we denote by \(\LIP(\X)\) the space of all real-valued Lipschitz functions on \(\X\),
and by \(\LIP_b(\X)\) its subspace consisting of bounded Lipschitz functions. The Lipschitz constant of \(f\in\LIP(\X)\)
on a set \(E\subseteq\X\) will be denoted by \(\Lip(f;E)\), and we will shorten \(\Lip(f)\coloneqq\Lip(f;\X)\).
The \textbf{asymptotic slope} \(\lip_a(f)\colon\X\to[0,\Lip(f)]\) of \(f\) is given by
\begin{equation}\label{eq:def_lip_a}
\lip_a(f)(x)\coloneqq\inf_{r>0}\Lip(f;B_r(x))\quad\text{ for every }x\in\X,
\end{equation}
where \(B_r(x)\) stands for the open ball of center \(x\) and radius \(r\). The following asymptotic-slope-preserving
Lipschitz extension result was proved by Di Marino--Gigli--Pratelli \cite[Theorem 1.1]{DMGP20}.
\begin{theorem}\label{thm:extension_lip_a}
Let \((\X,\sfd)\) be a metric space. Let \(E\subseteq\X\) be a non-empty set. Let \(f\in\LIP(E)\) and \(\varepsilon>0\)
be given. Then there exists an extension \(\bar f\in\LIP(\X)\) of \(f\) such that \(\Lip(\bar f)\leq\Lip(f)+\varepsilon\) and
\[
\lip_a(\bar f)(x)=\lip_a(f)(x)\quad\text{ for every }x\in E.
\]
Moreover, if \(f\in\LIP_b(E)\), then its extension \(\bar f\) can be chosen so that \(\inf_E f\leq\bar f\leq\sup_E f\) on \(\X\).
\end{theorem}
\subsection{Banach spaces}
We recall those concepts and results about Banach spaces that are relevant to this note, referring to \cite{HJ14,Casazza01}
for a thorough account of these topics. We denote by \(\B^*\) the dual of a Banach space \(\B\), by \(C^\infty(\B)\) the
space of smooth functions \(f\colon\B\to\R\), and by \(\d_x f\in\B^*\) the Fr\'{e}chet differential of \(f\) at \(x\in\B\).
Observe that \(\|\d_x f\|_{\B^*}=\lip_a(f)(x)\). A distinguished subalgebra of \(C^\infty(\B)\) is the space of
\textbf{smooth cylindrical functions} on \(\B\), which is given by
\begin{equation}\label{eq:def_Cyl}
\Cyl(\B)\coloneqq\bigg\{g\circ p\colon\B\to\R\;\bigg|\;\begin{array}{ll}
\mathbb V\text{ finite-dimensional Banach space},\,p\colon\B\to\mathbb V\\
\text{bounded linear operator},\,g\in C^\infty(\mathbb V)\cap\LIP_b(\mathbb V)
\end{array}\bigg\}.
\end{equation}
A Banach space \(\B\) is said to have the \textbf{metric approximation property} if the following holds: given any
\(\varepsilon>0\) and a compact set \(K\subseteq\B\), there exists a finite-rank linear operator \(p\colon\B\to\B\) with operator
norm at most \(1\) such that \(p|_K\) is an \(\varepsilon\)-approximation of the identity, meaning that
\[
\|p(x)-x\|_\B\leq\varepsilon\quad\text{ for every }x\in K.
\]
In the above, ``\(K\) compact'' can be replaced by ``\(K\) finite''. Indeed, if \(F\) is an \(\varepsilon\)-net
for \(K\) compact and \(p|_F\) is an \(\varepsilon\)-approximation of the identity, then \(p|_K\) is a
\(3\varepsilon\)-approximation of the identity.
\medskip

Albeit well-known to the experts, for the reader's usefulness we give a self-contained proof of the fact that every Banach
space can be embedded linearly and isometrically into a Banach space having the metric approximation property, see Corollary
\ref{cor:embed_in_MAP}. To this aim, we will work with the space \(\ell_\infty(S)\) of all bounded functions from \(S\) to
\(\R\), where \(S\neq\varnothing\) is an arbitrary set. The elements of \(\ell_\infty(S)\) will be denoted by
\(a=(a_x)_{x\in S}\subseteq\R\). The space \(\ell_\infty(S)\) is a Banach space if endowed with the pointwise vector space
operations and with the norm \(\|a\|_{\ell_\infty(S)}\coloneqq\sup_{x\in S}|a_x|\).
\begin{remark}\label{rmk:embed_into_MAP}{\rm
\emph{Any given Banach space \(\B\) embeds linearly and isometrically into \(\ell_\infty(\mathbb S_{\B^*})\)},
where \(\mathbb S_{\B^*}\) denotes the unit sphere \(\{\omega\in\B^*\,:\,\|\omega\|_{\B^*}=1\}\) of \(\B^*\).
Indeed, \(\B\ni v\mapsto(\omega(v))_{\omega\in\mathbb S_{\B^*}}\in\ell_\infty(\mathbb S_{\B^*})\) is a linear
isometry thanks to the Hahn--Banach extension theorem.
\fr}\end{remark}
\begin{lemma}\label{lem:ell_infty_MAP}
Let \(S\neq\varnothing\) be a given set. Then \(\ell_\infty(S)\) has the metric approximation property.
\end{lemma}
\begin{proof}
Fix any \(\varepsilon>0\) and \(a^1,\ldots,a^n\in\ell_\infty(S)\). We can partition \(S\) into non-empty sets
\(S_1,\ldots,S_k\) such that \(\{(a^1_x,\ldots,a^n_x)\,:\,x\in S_j\}\subseteq\R^n\) has diameter
at most \(\varepsilon\) for every \(j=1,\ldots,k\). Fix also some \((x_1,\ldots,x_k)\in S_1\times\ldots\times S_k\).
Given any \(a\in\ell_\infty(S)\), we define \(p(a)=(p(a)_x)_{x\in S}\in\ell_\infty(S)\) as
\[
p(a)_x\coloneqq a_{x_j}\quad\text{ for every }j=1,\ldots,k\text{ and }x\in S_j.
\]
The resulting map \(p\colon\ell_\infty(S)\to\ell_\infty(S)\) is a finite-rank linear operator of operator norm \(1\).
Moreover, we have that \(\|p(a^i)-a^i\|_{\ell_\infty(S)}=\sup_{j=1,\ldots,k}\sup_{x\in S_j}|a^i_{x_j}-a^i_x|\leq\varepsilon\)
holds for every \(i=1,\ldots,n\).
\end{proof}

Combining Lemma \ref{lem:ell_infty_MAP} with Remark \ref{rmk:embed_into_MAP}, we obtain the following embedding result.
\begin{corollary}\label{cor:embed_in_MAP}
Every Banach space embeds linearly and isometrically into a Banach space having the metric approximation property.
\end{corollary}
\begin{remark}{\rm
Since spaces under consideration in metric measure geometry are often assumed to be separable, it might be
useful to observe that \emph{\(C([0,1])\) is a universal separable Banach space having the metric approximation
property}, where `universal separable Banach space' means that it is a separable Banach space wherein each separable Banach
space can be embedded linearly and isometrically. See the Banach--Mazur theorem \cite[Proposition 1.5]{BP75} and
\cite[Example 4.2]{Ryan02}.
\fr}\end{remark}
\section{The approximation result}\label{s:approx_result}
We now pass to the main result of this note. First, we make a preliminary observation.
\begin{remark}\label{rmk:convolution}{\rm
Let \((\mathbb V,\|\cdot\|)\) be a finite-dimensional Banach space and \(f\in\LIP_b(\mathbb V)\). Then for any given \(\varepsilon>0\)
there exists a function \(f_\varepsilon\in C^\infty(\mathbb V)\cap\LIP_b(\mathbb V)\) such that the following properties hold:
\begin{itemize}
\item \(\inf_{\mathbb V}f\leq f_\varepsilon\leq\sup_{\mathbb V}f\) on \(\mathbb V\) and \(\Lip(f_\varepsilon)\leq\Lip(f)\).
\item \(|f_\varepsilon(x)-f(x)|\leq\Lip(f)\varepsilon\) for every \(x\in\mathbb V\).
\item \(\lip_a(f_\varepsilon)(x)\leq\Lip(f;B_\varepsilon(x))\) for every \(x\in\mathbb V\).
\end{itemize}
This claim can be proved via a standard convolution argument, see e.g.\ \cite[Lemma 2.9]{Gig:Pas:21}.
\fr}\end{remark}
\begin{theorem}\label{thm:approx_with_cyl}
Let \(\B\) be a Banach space and \(E\subseteq\B\) a \(\sigma\)-compact set. Let \(f\in{\rm LIP}_b(\B)\) and \(\varepsilon>0\)
be given. Then there exists a sequence \((f_n)_n\subseteq\Cyl(\B)\) such that the following properties hold:
\begin{itemize}
\item[\(\rm i)\)] \(\inf_\B f\leq f_n\leq\sup_\B f\) on \(\B\) and \(\Lip(f_n)\leq\Lip(f)+\varepsilon\) for every \(n\in\N\).
\item[\(\rm ii)\)] \(\lim_n f_n(x)=f(x)\) for every \(x\in E\).
\item[\(\rm iii)\)] \(\varlimsup_n\|\d_x f_n\|_{\B^*}\leq\lip_a(f)(x)\) for every \(x\in E\).
\end{itemize}
\end{theorem}
\begin{proof}
Thanks to Corollary \ref{cor:embed_in_MAP}, we have that \(\B\) is a subspace of some Banach space \(\mathbb M\) having the metric approximation property.
By Theorem \ref{thm:extension_lip_a}, there exists an extension \(\bar f\in\LIP_b(\mathbb M)\) of \(f\) such that \(\inf_\B f\leq\bar f\leq\sup_\B f\) on \(\mathbb M\),
\(\Lip(\bar f)\leq\Lip(f)+\varepsilon\), and \(\lip_a(\bar f)(x)=\lip_a(f)(x)\) for every \(x\in\B\). Let us write the set \(E\) as \(\bigcup_{n\in\N}K_n\), for some
increasing sequence \((K_n)_n\) of compact subsets of \(\B\). Given any \(n\in\N\), we fix a finite-rank \(1\)-Lipschitz linear operator \(p_n\colon\mathbb M\to\mathbb M\) such that
\[
\|p_n(x)-x\|_{\mathbb M}\leq\frac{1}{n}\quad\text{ for every }x\in K_n.
\]
Since \(\mathbb V_n\coloneqq p_n(\mathbb M)\) is finite-dimensional and \(\bar f|_{\mathbb V_n}\in\LIP_b(\mathbb V_n)\), by Remark \ref{rmk:convolution} we know that
there exists a function \(g_n\in C^\infty(\mathbb V_n)\cap\LIP_b(\mathbb V_n)\), with \(\inf_{\mathbb M}\bar f\leq g_n\leq\sup_{\mathbb M}\bar f\) and \(\Lip(g_n)\leq\Lip(\bar f;\mathbb V_n)\),
such that \(|g_n-\bar f|\leq 1/n\) and \(\lip_a(g_n)\leq\Lip(\bar f;B_{1/n}(\cdot)\cap\mathbb V_n)\) on \(\mathbb V_n\). Define \(f_n\coloneqq g_n\circ p_n|_{\mathbb B}\in\Cyl(\B)\).
Given that \(\inf_\B f=\inf_{\mathbb M}\bar f\leq f_n\leq\sup_{\mathbb M}\bar f=\sup_\B f\) and \(\Lip(f_n)\leq\Lip(g_n)\leq\Lip(\bar f)\leq\Lip(f)+\varepsilon\),
item i) is proved. Moreover, given any \(n,k\in\N\) with \(n\leq k\) and any \(x\in K_n\), we can estimate
\[
|f_k(x)-f(x)|\leq|(g_k-\bar f)(p_k(x))|+|\bar f(p_k(x))-f(x)|\leq\frac{1}{k}+\Lip(\bar f)\|p_k(x)-x\|_{\mathbb M}\leq\frac{\Lip(\bar f)+1}{k},
\]
which yields \(\lim_k f_k(x)=f(x)\) for every \(x\in E\), proving ii). Finally, if \(n\in\N\) and \(x\in K_n\), then
\[
\lip_a(f_n)(x)\leq\lip_a(g_n)(p_n(x))\leq\Lip(\bar f;B_{1/n}(p_n(x)))\leq\Lip(\bar f;B_{2/n}(x))
\]
(here we use the fact that \(B_{1/n}(p_n(x))\subseteq B_{2/n}(x)\), as \(\|p_n(x)-x\|_{\mathbb M}\leq 1/n\)), whence it follows that
\[
\varlimsup_n\lip_a(f_n)(x)\leq\lim_n\Lip(\bar f;B_{2/n}(x))=\lip_a(\bar f)(x)=\lip_a(f)(x)\quad\text{ for every }x\in E.
\]
Consequently, also item iii) is proved (since \(\lip_a(f_n)=\|\d_\cdot f_n\|_{\B^*}\)). The statement is achieved.
\end{proof}

Let \(\mu\geq 0\) be a finite Borel measure on \(\B\) concentrated on a \(\sigma\)-compact set, e.g.\ \(\mu\) is Radon.
(Assuming the cardinality of any set is an Ulam number -- which is consistent with the ZFC set theory -- each
finite Borel measure is concentrated on a \(\sigma\)-compact set; see \cite[Lemma 2.9]{AmbrosioKirchheim00}.)
Then Theorem \ref{thm:approx_with_cyl} implies that the following holds: \emph{given any \(f\in{\rm LIP}_b(\B)\)
and \(\varepsilon>0\), there exists a sequence \((f_n)_n\subseteq\Cyl(\B)\), with \(\inf_\B f\leq f_n\leq\sup_\B f\) and \(\sup_n\Lip(f_n)\leq\Lip(f)+\varepsilon\), such that}
\[
\lim_n f_n(x)=f(x),\quad\varlimsup_n\|\d_x f_n\|_{\B^*}\leq\lip_a(f)(x)\quad\text{ for }\mu\text{-a.e.\ }x\in\B.
\]
By applying the dominated convergence theorem, one can readily deduce that \emph{\(\Cyl(\B)\) is strongly dense
in \(L^p(\mu)\) for every \(p\in[1,\infty)\)}. The latter property follows also from \cite[Lemma 2.1.27]{Sav19},
since \(\Cyl(\B)\) is compatible (in the sense of \cite[Definition 2.1.17]{Sav19}), cf.\ \cite[Example 2.1.19]{Sav19}.
\section{Applications to Sobolev and BV calculus}
As a useful consequence of Theorem \ref{thm:approx_with_cyl}, it is possible to obtain density results of the following kind:
\emph{if Lipschitz functions are `dense in energy' in some given functional space over a weighted Banach space,
then cylindrical functions are dense in energy as well.} Let us now illustrate two such examples, which motivated our interest
in the approximation result proved in Theorem \ref{thm:approx_with_cyl}.
\medskip

Let \(\B\) be a separable Banach space and let \(\mu\geq 0\) be a finite Borel measure on \(\B\). Following
\cite{AmbrosioGigliSavare11,AmbrosioGigliSavare11-3} (also \cite{Cheeger00}) and \cite{Mir03},
we consider the following notions of Sobolev and BV spaces over \((\B,\mu)\).
\begin{itemize}
\item Given \(p\in[1,\infty)\), a function \(f\in L^p(\mu)\) belongs to the \textbf{Sobolev space} \(W^{1,p}(\B,\mu)\)
if there exist \((f_n)_n\subseteq\LIP_b(\B)\) and \(G\in L^p(\mu)^+\) such that \(f_n\to f\) and
\(\lip_a(f_n)\to G\) in \(L^p(\mu)\). The unique \(\mu\)-a.e.\ minimal such function \(G\) is called the \textbf{minimal relaxed slope} \(|{\rm D}f|\) of \(f\).
The following fact holds: if a given sequence \((f_n)_n\subseteq\LIP_b(\B)\) satisfies \(f_n\to f\) in \(L^p(\mu)\) and \(\varlimsup_n\|\lip_a(f_n)\|_{L^p(\mu)}\leq\||{\rm D}f|\|_{L^p(\mu)}\),
then \(\lip_a(f_n)\to|{\rm D}f|\) in \(L^p(\mu)\).
\item A function \(f\in L^1(\mu)\) belongs to the space \(BV(\B,\mu)\) of \textbf{functions of bounded variation} if
there exists \((f_n)_n\subseteq\LIP_b(\B)\) such that \(f_n\to f\) in \(L^1(\mu)\) and \(\varliminf_n\int_\B\lip_a(f_n)\,\d\mu<+\infty\).
Given any \(f\in BV(\B,\mu)\), there exists a unique finite Borel measure \(|{\bf D}f|\) on \(\B\) such that
\[
|{\bf D}f|(\Omega)=\inf\bigg\{\varliminf_n\int_\Omega\lip_a(f_n)\,\d\mu\;\bigg|\;\LIP_{loc}(\Omega)\cap L^1(\mu|_\Omega)\ni f_n\to f|_\Omega\text{ in }L^1(\mu|_\Omega)\bigg\}
\]
holds for every open set \(\Omega\subseteq\B\), where \(\LIP_{loc}(\Omega)\) denotes the space of all real-valued locally Lipschitz functions on \(\Omega\). Moreover,
we have that \(|{\bf D}f|(\B)=\inf_{(f_n)_n}\varliminf_n\int_\B\lip_a(f_n)\,\d\mu\), where the infimum is taken among all sequences \((f_n)_n\subseteq\LIP_b(\B)\)
with \(f_n\to f\) in \(L^1(\mu)\).
\end{itemize}

Combining Theorem \ref{thm:approx_with_cyl} with a diagonal argument, we can easily obtain the following result.
\begin{theorem}\label{thm:dens_cyl}
Let \(\B\) be a separable Banach space. Let \(\mu\geq 0\) be a finite Borel measure on \(\B\).
\begin{itemize}
\item[\(\rm i)\)] Let \(p\in[1,\infty)\) and \(f\in W^{1,p}(\B,\mu)\) be given. Then there exists a sequence
\((f_n)_n\subseteq\Cyl(\B)\) such that \(f_n\to f\) and \(\|\d_\cdot f_n\|_{\B^*}\to|{\rm D}f|\) in \(L^p(\mu)\).
\item[\(\rm ii)\)] Let \(f\in BV(\B,\mu)\) be given. Then there exists a sequence \((f_n)_n\subseteq\Cyl(\B)\) such
that \(f_n\to f\) in \(L^1(\mu)\) and \(\int_\B\|\d_\cdot f_n\|_{\B^*}\,\d\mu\to|{\bf D}f|(\B)\). In particular, it holds that
\(\|\d_\cdot f_n\|_{\B^*}\mu\rightharpoonup|{\bf D}f|\) weakly, i.e.\ in duality with the space of real-valued bounded continuous functions on \(\B\).
\end{itemize}
\end{theorem}
\begin{proof}
Let us first prove i). Given any \(n\in\N\), take \(g_n\in\LIP_b(\B)\) such that
\(\|g_n-f\|_{L^p(\mu)}\leq 1/n\) and \(\|\lip_a(g_n)-|{\rm D}f|\|_{L^p(\mu)}\leq 1/n\). Applying Theorem \ref{thm:approx_with_cyl}
(and the paragraph after it), we can find a sequence \((f_n^k)_k\subseteq\Cyl(\B)\) with
\(\sup_k(\|f_n^k\|_{L^\infty(\mu)}+\Lip(f_n^k))<+\infty\) such that \(\lim_k f_n^k(x)=g_n(x)\) and
\(\varlimsup_k\|\d_x f_n^k\|_{\B^*}\leq\lip_a(g_n)(x)\) for \(\mu\)-a.e.\ \(x\in\B\). Applying the dominated convergence theorem and
the reverse Fatou lemma, we can find \(k(n)\in\N\) such that \(f_n\coloneqq f_n^{k(n)}\) satisfies \(\|f_n-g_n\|_{L^p(\mu)}\leq 1/n\)
and \(\|\|\d_\cdot f_n\|_{\B^*}\|_{L^p(\mu)}\leq\|\lip_a(g_n)\|_{L^p(\mu)}+1/n\). In particular, we have that \(\|f_n-f\|_{L^p(\mu)}\leq 2/n\)
and \(\|\|\d_\cdot f_n\|_{\B^*}\|_{L^p(\mu)}\leq\||{\rm D}f|\|_{L^p(\mu)}+2/n\), thus proving that i) holds.
Let us now pass to the verification of ii). The first claim can be proved by slightly modifying the proof of i),
choosing \(g_n\in\LIP_b(\B)\) with \(|\int_\B\lip_a(g_n)\,\d\mu-|{\bf D}f|(\B)|\leq 1/n\). For the second claim, notice that \(\nu_n\coloneqq\|\d_\cdot f_n\|_{\B^*}\mu\)
satisfies \(|{\bf D}f|(\Omega)\leq\varliminf_n\nu_n(\Omega)\) for every \(\Omega\subseteq\B\) open and \(\nu_n(\B)\to|{\bf D}f|(\B)\), so that \(\nu_n\rightharpoonup|{\bf D}f|\) weakly.
\end{proof}

\begin{thebibliography}{10}

\bibitem{ADM2014}
L.~Ambrosio and S.~Di~Marino.
\newblock Equivalent definitions of {BV} space and of total variation on metric measure spaces.
\newblock {\em Journal of Functional Analysis}, 266(7):4150 -- 4188, 2014.

\bibitem{AmbrosioGigliSavare11-3}
L.~Ambrosio, N.~Gigli, and G.~Savar{\'e}.
\newblock Density of {L}ipschitz functions and equivalence of weak gradients in metric measure spaces.
\newblock {\em Rev. Mat. Iberoam.}, 29(3):969--996, 2013.

\bibitem{AmbrosioGigliSavare11}
L.~Ambrosio, N.~Gigli, and G.~Savar{\'e}.
\newblock Calculus and heat flow in metric measure spaces and applications to spaces with {R}icci bounds from below.
\newblock {\em Invent. Math.}, 195(2):289--391, 2014.

\bibitem{AILP24}
L.~Ambrosio, T.~Ikonen, D.~Lu\v{c}i\'{c}, and E.~Pasqualetto.
\newblock Metric {S}obolev spaces {I}: equivalence of definitions.
\newblock Preprint, arXiv:2404.11190, 2024.

\bibitem{AmbrosioKirchheim00}
L.~Ambrosio and B.~Kirchheim.
\newblock Currents in metric spaces.
\newblock {\em Acta Math.}, 185(1):1--80, 2000.

\bibitem{BBF99}
G.~Bellettini, G.~Bouchitt{\'e}, and I.~Fragal{\`a}.
\newblock {B}{V} functions with respect to a measure and relaxation of metric integral functionals.
\newblock {\em Journal of convex analysis}, 6:349--366, 1999.

\bibitem{BP75}
C.~Bessaga and A.~Pe{\l}czy\'{n}ski.
\newblock {\em Selected topics in infinite-dimensional topology}.
\newblock Monografie matematyczne. Polska Akademia Nauk. Instytut Matematyczny, 1975.

\bibitem{BB22}
K.~Bo{\l}botowski and G.~Bouchitt{\'e}.
\newblock Optimal design versus maximal {M}onge-{K}antorovich metrics.
\newblock {\em Archive for Rational Mechanics and Analysis}, 243:1--76, 2022.

\bibitem{BL22}
K.~Bo{\l}botowski and T.~Lewi{\'n}ski.
\newblock Setting the {F}ree {M}aterial {D}esign problem through the methods of optimal mass distribution.
\newblock {\em Calculus of Variations and Partial Differential Equations}, 61:61--76, 2022.

\bibitem{BB01}
G.~Bouchitt{\'e} and G.~Buttazzo.
\newblock Characterization of optimal shapes and masses through {M}onge-{K}antorovich equation.
\newblock {\em Journal of the European Mathematical Society}, 3:139--168, 2001.

\bibitem{BBS97}
G.~Bouchitt\'{e}, G.~Buttazzo, and P.~Seppecher.
\newblock Energies with respect to a measure and applications to low dimensional structures.
\newblock {\em Calc. Var. Partial Differential Equations}, 5:37--54, 1997.

\bibitem{Casazza01}
P.~Casazza.
\newblock Approximation properties.
\newblock In {\em Handbook of the Geometry of Banach Spaces}, volume~1, pages 271--316, 2001.

\bibitem{Cheeger00}
J.~Cheeger.
\newblock Differentiability of {L}ipschitz functions on metric measure spaces.
\newblock {\em Geom. Funct. Anal.}, 9(3):428--517, 1999.

\bibitem{DGZ93}
R.~Deville, G.~Godefroy, and V.~Zizler.
\newblock {\em Smoothness and renormings in {B}anach spaces}.
\newblock Pitman monographs and surveys in pure and applied mathematics 64. Longman Scientific \& Technical, 1993.

\bibitem{DiMarino14}
S.~Di~Marino.
\newblock {S}obolev and {B}{V} spaces on metric measure spaces via derivations and integration by parts.
\newblock Preprint, arXiv:1409.5620, 2014.

\bibitem{DMGP20}
S.~Di~Marino, N.~Gigli, and A.~Pratelli.
\newblock Global {L}ipschitz extension preserving local constants.
\newblock {\em Rend. Lincei Mat. Appl.}, 31:757--765, 2020.

\bibitem{EB20}
S.~Eriksson-Bique.
\newblock Density of {L}ipschitz functions in energy.
\newblock {\em Calc. Var. Partial Differential Equations}, 62(2), 2023.

\bibitem{FHHMZ11}
M.~Fabian, P.~Habala, P.~H{\'a}jek, V.~Montesinos, and V.~Zizler.
\newblock {\em Banach {S}pace {T}heory}.
\newblock CMS Books in Mathematics. Springer New York, 2011.

\bibitem{FMZ06}
M.~Fabian, V.~Montesinos, and V.~Zizler.
\newblock Smoothness in {B}anach spaces. {S}elected problems.
\newblock {\em Rev. R. Acad. Cienc. Exactas F\'{i}s. Nat., Ser. A Mat.}, 100(1--2):101--125, 2006.

\bibitem{FornasierSavareSodini22}
M.~Fornasier, G.~Savar\'{e}, and G.~E. Sodini.
\newblock Density of subalgebras of {L}ipschitz functions in metric {S}obolev spaces and applications to {W}asserstein {S}obolev spaces.
\newblock {\em J. Funct. Anal.}, 285(11):110153, 2023.

\bibitem{GL22}
M.~S. Gelli and D.~Lu\v{c}i\'{c}.
\newblock A note on {B}{V} and $1$-{S}obolev functions on the weighted {E}uclidean space.
\newblock {\em Atti Accad. Naz. Lincei Cl. Sci. Fis. Mat. Natur.}, 33:757--794, 2022.

\bibitem{Gig:Pas:21}
N.~Gigli and E.~Pasqualetto.
\newblock Behaviour of the reference measure on $\sf{R}{C}{D}$ spaces under charts.
\newblock {\em Commun. Anal. Geom.}, 29(6):1391--1414, 2021.

\bibitem{HM17}
O.~A. Hafsa and J.-P. Mandallena.
\newblock {$\Gamma$}-convergence of nonconvex integrals in {C}heeger--{S}obolev spaces and homogenization.
\newblock {\em Advances in Calculus of Variations}, 10:381--405, 2017.

\bibitem{HJ14}
P.~H\'{a}jek and M.~Johanis.
\newblock {\em Smooth analysis in {B}anach spaces}.
\newblock De Gruyter Series in Nonlinear Analysis and Applications. De Gruyter, Berlin, 2014.

\bibitem{Louet14}
J.~Louet.
\newblock Some results on {S}obolev spaces with respect to a measure and applications to a new transport problem.
\newblock {\em J. Math. Sci. (N.Y.)}, 196:152--164, 2014.

\bibitem{LP24}
D.~Lu{\v{c}}i{\'c} and E.~Pasqualetto.
\newblock Yet another proof of the density in energy of {L}ipschitz functions.
\newblock {\em manuscripta mathematica}, 175:421--438, 2024.

\bibitem{LPR21}
D.~Lu\v{c}i\'{c}, E.~Pasqualetto, and T.~Rajala.
\newblock Characterisation of upper gradients on the weighted {E}uclidean space and applications.
\newblock {\em Ann. Mat. Pura Appl. (4)}, 200:2473--2513, 2021.

\bibitem{Martio16}
O.~Martio.
\newblock Functions of bounded variation and curves in metric measure spaces.
\newblock {\em Advances in Calculus of Variations}, 9(4):305--322, 2016.

\bibitem{Mir03}
M.~{Miranda Jr.}
\newblock Functions of bounded variation on ``good'' metric spaces.
\newblock {\em Journal de Math{\'e}matiques Pures et Appliqu{\'e}es}, 82(8):975--1004, 2003.

\bibitem{NobiliPasqualettoSchultz23}
F.~Nobili, E.~Pasqualetto, and T.~Schultz.
\newblock On master test plans for the space of {BV} functions.
\newblock {\em Adv. Calc. Var.}, 16(4):1061--1092, 2023.

\bibitem{Pa:Ra:23}
E.~Pasqualetto and T.~Rajala.
\newblock Vector calculus on weighted reflexive {B}anach spaces.
\newblock Preprint, arXiv:2306.03684, 2023.

\bibitem{Ryan02}
R.~A. Ryan.
\newblock {\em Introduction to {T}ensor {P}roducts of {B}anach {S}paces}.
\newblock Springer Monographs in Mathematics. Springer London, 2002.

\bibitem{Sav19}
G.~Savar{\'e}.
\newblock Sobolev spaces in extended metric-measure spaces.
\newblock In {\em New Trends on Analysis and Geometry in Metric Spaces}, pages 117--276. Springer International Publishing, 2022.

\bibitem{Shanmugalingam00}
N.~Shanmugalingam.
\newblock Newtonian spaces: an extension of {S}obolev spaces to metric measure spaces.
\newblock {\em Rev. Mat. Iberoam.}, 16(2):243--279, 2000.

\bibitem{Sodini22}
G.~E. Sodini.
\newblock The general class of {W}asserstein {S}obolev spaces: density of cylinder functions, reflexivity, uniform convexity and {C}larkson's inequalities.
\newblock {\em Calc. Var. Partial Differential Equations}, 62(7):1--41, 2023.

\bibitem{Zhi00}
V.~V. Zhikov.
\newblock On an extension of the method of two-scale convergence and its applications.
\newblock {\em Sb. Math.}, 191:973--1014, 2000.

\bibitem{Zhi02}
V.~V. Zhikov.
\newblock Homogenization of elasticity problems on singular structures.
\newblock {\em Izv. Math.}, 66:299--365, 2002.

\end{thebibliography}
\end{document}